\documentclass{amsart}
\usepackage{amsmath, amsthm, amssymb}

\newtheorem{theorem}{Theorem}[section]
\newtheorem{lemma}[theorem]{Lemma}
\newtheorem{Proposition}[theorem]{Proposition}
\theoremstyle{definition}
\newtheorem{definition}[theorem]{Definition}

\theoremstyle{remark}

\numberwithin{equation}{section}

\begin{document}

\title{Codimension One Structurally Stable Chain Classes}


\author{Xiao Wen }
\address{School of Mathematics and System Science, Beihang University, Beijing 100191, China}
\curraddr{} \email{wenxiao@buaa.edu.cn}
\thanks{}

\author{Lan Wen}
\address{LMAM, School of Mathematical Sciences, Peking University, Beijing 100871, China}
\curraddr{} \email{lwen@math.pku.edu.cn}
\thanks{This work is partially supported by the Balzan Research
Project of J. Palis. XW is supported by NSFC 11301018. LW is
supported by NSFC 11231001.}



\date{}

\dedicatory{}

\begin{abstract}
The well known stability conjecture of Palis and Smale states that
if a diffeomorphism is structurally stable then the chain recurrent
set is hyperbolic. It is natural to ask if this type of results is
true for an individual chain class, that is, whether or not every
structurally stable chain class is hyperbolic. Regarding the notion
of structural stability, there is a subtle difference between the
case of a whole system and the case of an individual chain class.
The later case is more delicate and contains additional
difficulties. In this paper we prove a result of this type for the
later case, with an additional assumption of codimension 1.
Precisely, let $f$ be a diffeomorphism of a closed manifold $M$ and
$p$ be a hyperbolic periodic point of $f$ of index 1 or $\dim M-1$.
We prove if the chain class  of $p$ is structurally stable then it
is hyperbolic. Since the chain class of $p$ is not assumed in
advance to be locally maximal, and since the counterpart of it for
the perturbation $g$ is defined not canonically but indirectly
through the continuation  $p_g$ of $p$, the proof is quite delicate.
\end{abstract}

\maketitle

\section{Introduction}
Let $M$ be a compact $C^\infty$ Riemannian manifold without
boundary, and $f:M\to M$ be a diffeomorphism.  Denote ${\rm
Diff}(M)$ the space of diffeomorphisms of $M$ with the
$C^1$-topology.

It is  understood that the main dynamics of a system appears in the
part that exhibits certain recurrence, as it contains the long run
behavior of all orbits. The most general notion of recurrence is the
so called chain recurrence. Its definition is standard, but we
include it here for completeness.

Let  $\delta>0$ be given. A finite sequence
$\{x_i\}_{i=0}^{n}\subset M$ is called a {\it $\delta$-pseudo orbit}
of $f$ if $d(f(x_i),x_{i+1})<\delta$ for all $0\leq i \leq n-1$,
where $d$ is the distance on $M$ induced by the Riemannian metric.
For two points $x,y \in M$, we write $x\dashv y$ if, for any $\delta
>0$, there is a $\delta$-pseudo orbit   of $f$ going from $x$ to $y$, that is,
there is  a $\delta$-pseudo orbit $\{x_i\}_{i=0}^{n}$, where $n$
depends on $\delta$, such that $x_0=x $ and $x_{n}=y$.
 A point $x\in M$ is called
 {\it chain recurrent} if $x\dashv x$.  Thus a chain recurrent point is one with
  a (very weak) recurrence in the sense of pseudo orbits. The set of chain recurrent
 points of $f$ is called the {\it chain recurrent  set}
  of $f$, denoted by ${{\rm CR}}(f)$. It is easy to see that ${{\rm CR}}(f)$
  is closed and $f({{\rm CR}}(f))={{\rm CR}}(f)$. Clearly,
  $$\overline{{\rm Per}(f)}\subset \Omega(f)\subset {{\rm CR}}(f),$$ where ${\rm Per}(f)$ is the
  set of periodic points and $\Omega(f)$ is the non-wandering set of
  $f$. By Conley \cite{Con}, any point that is not chain recurrent must be in the basin of
  some  attracting set subtracting the attracting set itself, hence exhibits no recurrence of any type. Thus  chain recurrence
  is the most general version of recurrence.

 An  important notion in dynamical systems  coming from
 Physics and Mechanics  is the so called structural stability.
 Precisely, a diffeomorphism $f$ is  {\it structurally stable} if there
is a $C^1$ neighborhood ${\mathcal U}$ of $f$ in ${\rm Diff}(M)$
such that, for every $g\in {\mathcal U}$, there is a homeomorphism
$h: M\to M$ such that $h\circ f=g\circ h$. Since such a
homeomorphism $h$ preserves orbits, a structurally stable system is
one that has robust dynamics, that is, one whose orbital structure
remains unchanged under perturbations.

The non-recurrent part of dynamical systems is fairly robust with
respect to perturbations. But the recurrent part is fragile. To
survive from perturbations, it needs  the condition of (various
versions of) hyperbolicity. For instance, a single periodic orbit is
structurally stable if and only if it is hyperbolic, meaning no
eigenvalue of modulus 1. For the whole system $f$ to be structurally
stable, a crucial condition needed is that ${\rm CR}(f)$, the set
that captures all the recurrence, is a hyperbolic set. Recall a
compact invariant set $\Lambda\subset M$ of $f$ is called {\it
hyperbolic} if, for
 each $x\in \Lambda$, the tangent space $T_x M$  splits into
 $ T_x M=E^s(x)\oplus E^u(x)$ such that
$$ Df(E^s(x))=E^s(f(x)), \ Df(E^u(x))=E^u(f(x))$$
and, for some constants $C\ge 1$ and $0<\lambda<1$,
$$ |Df^n(v)|\le C\lambda^n |v|, \ \forall x\in \Lambda, v\in E^s(x), n\ge 0,$$
$$ |Df^{-n}(v)|\le C\lambda^n |v|, \ \forall x\in \Lambda, v\in E^u(x), n\ge 0.$$

Briefly, a hyperbolic set is one at which tangent vectors split into
two directions, contracting and expanding upon iterates,
respectively, with uniform exponential rates. This definition
extends the hyperbolicity condition from a single periodic orbit to
a general compact invariant set. It is closely related to structural
stability. Indeed, the following remarkable result, known as the
stability conjecture of Palis and Smale \cite{PS1}, is fundamental
to dynamical systems:

\bigskip
\noindent{\bf Theorem} (Ma\~n\'e \cite{Man2}). \ {\it
 If a diffeomorphism $f$ is structurally stable then
${\rm CR}(f)$ is hyperbolic.}

\bigskip

In this paper we consider a localized  and more delicate version of
structural stability.  It is for an individual ``basic piece" of the
dynamics, rather than the whole system. Note that, restricted to
${{\rm CR}}(f)$, the relation $x\sim y$ (meaning $x\dashv y$ and
$y\dashv x$) is an equivalence relation. The equivalence classes are
called {\it chain classes} of $f$, which are each compact and
invariant under $f$. Any chain class can not be decomposed into two
disjoint compact invariant sets, hence is regarded as a basic piece
of the system. Generally, a diffeomorphism may have infinitely many
chain classes, a phenomenon that causes a great deal of complexity
of the dynamics. For any periodic point of $f$, denote $C_f(p)$ the
(unique) chain class of $f$ that contains $p$.

A hyperbolic periodic point has its natural ``continuation" under
perturbations. Precisely, let $p\in M$ be a hyperbolic periodic
point of $f$ of period $k$. Then there exist a compact neighborhood
$U$ of ${\rm Orb}(p)$ in $M$ and a $C^1$-neighborhood
$\mathcal{U}(f)$ of $f$ such that  for any $g\in \mathcal{U}(f)$,
the maximal invariant set
$$\bigcap_{n=-\infty}^{\infty} g^n(U)$$
of $g$ in $U$ consists of a single periodic orbit $O_g$ of $g$ of
the same period as $p$, which is hyperbolic with ${\rm
Ind}(O_g)={\rm Ind}(p)$. Here ${\rm Ind}(p)$ denotes the {\it index}
of $p$, which is the dimension of the stable manifold of $p$.  The
neighborhood $U$ can be chosen to be the union of $k$ arbitrarily
small disjoint balls, each containing exactly one point of ${\rm
Orb}(p)$ and one point of $O_g$. This identifies the {\it
continuation} $p_g$ of $p$ under $g$. Thus the notion of
continuation $p_g$ of $p$ is defined for $g$ sufficiently close to
$f$.

However, there is no ``continuation" well-defined for a general
compact invariant set. Indeed, for a  general compact invariant set
$\Lambda$ of $f$ (not a specific one such as $\Omega(f)$, ${{\rm
CR}}(f)$, etc.), there is no canonical way to define the
``counterpart" of $\Lambda$ for $g$ that is near $f$. Consequently,
there is no canonical way to define such a general $\Lambda$ to be
``structurally stable". Nevertheless for the case of a chain
recurrent class that contains a hyperbolic periodic point $p$, there
is an indirect way as follows to define its structural stability,
through the continuation $p_g$ of $p$. Let $C_g(p_g)$ denote the
(unique) chain class of $g$ that contains $p_g$.

\begin{definition}
Let $p$ be a hyperbolic  periodic point of $f$.  We say that
$C_f(p)$ is $C^1$-{\it structurally stable} if there is a
neighborhood $\mathcal{U}$ of $f$ in ${\rm Diff}(M)$ such that, for
every $g\in \mathcal{U}$, there is a homeomorphism $h: C_f(p)\to
C_g(p_g)$ such that $h\circ f|_{C_f(p)}=g\circ h|_{C_f(p)}$, where
$p_g$ is the continuation of $p$.
\end{definition}

Note that, while $h$ in this definition preserves periodic points of
$C_f(p)$, it is not clear if it preserves individual continuations.
For instance, it is not clear if $h(p)=p_g$. Indeed, such an
``indirect" definition of structural stability makes the proof of
the following main theorem of this paper quite delicate:

\bigskip
\noindent {\bf Theorem A.} Let $f$ be a diffeomorphism of $M$ and
$p$ be a hyperbolic periodic point of $f$ of index $1$ or ${\rm dim}
M-1$. If the chain class $C_f(p)$ of $p$ is structurally stable,
then $C_f(p)$ is hyperbolic.

\bigskip

This result is in the spirit of the  stability conjecture, but more
delicate as just indicated. In particular, $C_f(p)$ is not assumed
in advance to be locally maximal (meaning being the maximal
invariant set in a neighborhood of itself),  hence periodic orbits
that are proved  to exist in a neighborhood of $C_f(p)$ are hardly
identified to be actually inside $C_f(p)$.  This is a serious
difficulty that appears in the proof.

 A special strategy we will use to prove Theorem A is  first
to prove the theorem for a generic  $f$, that is, for $f$ in a
residual family of diffeomorphisms. Most part of this paper will be
devoted to this special case. Then, for such a generic $f$, $C_f(p)$
is shadowable, because  hyperbolicity implies the shadowing
property. Since a topological conjugacy, even one that is defined on
a chain class only, preserves the shadowing property, by picking up
a generic diffeomorphism near $f$, we see that a structurally stable
chain class $C_f(p)$ is robustly shadowable. But, according to a
previous result of X. Wen et. al. \cite{WGW}, a robustly shadowable
chain class must be hyperbolic. This will be the way how Theorem A
is proved.

\section{Periodic points in $C_f(p)$}

For a hyperbolic periodic point $p$ of $f$, denote by $H(p, f)$ the
{\it homoclinic class} of $p$, that is, the closure of the set of
transverse homoclinic points of ${\rm Orb}(p)$. We say that two
hyperbolic periodic points $p$ and $q$ of $f$ are {\it
homoclinically related} if $W^s({\rm Orb(p)})\pitchfork W^u({\rm
Orb}(q))\not=\emptyset$ and $W^s({\rm Orb}(q))\pitchfork W^u({\rm
Orb}(p))\not=\emptyset$. Note that two hyperbolic periodic points
that are homoclinically related have the same index. A homoclinic
class $H(p, f)$ contains a dense subset of periodic points that are
homoclinically related to $p$.

In the main body of this paper we will work with a generic
diffeomorphism $f$ with the following properties:

\begin{Proposition}\label{Generic} There is a residual
subset $\mathcal{R}\subset {\rm Diff}(M)$ such that every
$f\in\mathcal{R}$ satisfies the following conditions:
\begin{itemize}
\item [1.] $f$ is Kupka-Smale, meaning periodic points of $f$ are each hyperbolic
and their stable and unstable manifolds meet transversally $($see
\cite{PM}).

\item [2.] Any chain  class of $f$ containing a hyperbolic periodic point
$p$ of $f$ equals $H(p, f)$ $($see \cite{BC}).

\item [3.] For any pair of
hyperbolic periodic points $p$ and $q$ of $f$, either $H(p, f) =
H(q, f)$ or $H(p, f)\cap H(q, f) =\emptyset$.

\item [4.] If two hyperbolic periodic points $p$ and $q$ of $f$ are in the same
topologically transitive set and ${\rm Ind}(p)\leq{\rm Ind}(q)$,
then $W^s({\rm Orb}(q), f)\pitchfork W^u({\rm Orb}(p), f)$ $\neq
\emptyset$ $($see \cite{GW}).

\item [5.] Every chain transitive set of $f$ is a Hausdorff limit of
periodic orbits of $f$ $($see \cite{Cro3}).

\end{itemize}
\end{Proposition}

Note that Item 3 is a consequence of Item 2. Also note that,
throughout this paper, the letter $\mathcal{R}$ will denote the
residual set described in this proposition.

\begin{Proposition}\label{Proposition1}
 Let $f\in\mathcal{R}$, and let $p\in M$ be a hyperbolic periodic point of $f$.
 If $C_f(p)$ is structural stable, then every periodic point $q\in
C_f(p)$ of $f$ is homoclinically related to $p$.
\end{Proposition}

\begin{proof}
We prove all periodic points of $f$ in $C_f(p)$ have the same index.
Suppose for contradiction there is a periodic point $p'\in C_f(p)$
with ${\rm Ind}(p')\neq {\rm Ind}(p)$. Let
$$k=\max\{\pi(p), \ \pi(p')\},$$
where $\pi$ denotes the period of a periodic point. Since $C_f(p)$
is structural stable, there is a $C^1$ neighborhood ${\mathcal U}$
of $f$ in ${\rm Diff}(M)$ such that, for any $g\in {\mathcal U}$,
there is a homeomorphism $h$ that conjugates $C_f(p)$ and
$C_g(p_g)$. For any $g\in {\mathcal U}$, denote ${\mathcal P}_k(g,
C_g(p_g))$ the set of (not necessarily hyperbolic) periodic orbits
of $g$  in $C_g(p_g)$ that has period less than or equal to $k$.
Then
$$h({\mathcal P}_k(f, C_f(p)))={\mathcal P}_k(g, C_g(p_g)).$$ Since $f$
is Kupka-Smale, ${\mathcal P}_k(f, C_f(p))$ is a finite set. Then
${\mathcal P}_k(g, C_g(p_g))$ has the same number of elements.

We need a topological version of heteroclinic cycle here. Precisely,
for a (not necessarily hyperbolic) periodic orbit $Q$ of $g$, define
the {\it stable manifold} of $Q$ to be
$$W^s(Q)=W^s(Q, g)=\{x\in M \ | \ d(g^n(x), Q)\to 0, n\to
+\infty\}.$$ Likewise for the {\it unstable manifold}  $W^u(Q)$.
Since $Q$ is not required to be hyperbolic, $W^s(Q)$ and $W^u(Q)$
are not necessarily differentiable manifolds. We say two not
necessarily hyperbolic periodic orbits $Q_1$ and $Q_2$ of $g$  form
a {\it heteroclinic cycle}, denoted $Q_1\sim Q_2$, if $W^s(Q_1)\cap
W^u(Q_2)\neq \emptyset$ and $W^s(Q_2)\cap W^u(Q_1)\neq \emptyset$.
Here it may not be meaningful to talk about transversality of these
intersections. The notion of heteroclinic cycle is standard, here we
just relax the requirement for the differentiability and
transversality. We will use this topological version of heteroclinic
cycle in the proof of this proposition only. Note that $Q_1$ and
$Q_2$ are in the same chain class because of the intersections
$W^s(Q_1)\cap W^u(Q_2)$ and $W^s(Q_2)\cap W^u(Q_1)$, and every point
in the intersections belongs to the same chain class. Thus the above
topological conjugacy $h$ preserves heteroclinic cycles. That is, if
$Q_1, Q_2\subset C_f(p)$ and $Q_1\sim Q_2$, then $h(Q_1)\sim
h(Q_2)$, and vise versa.

 Although the intersections in the heteroclinic cycles are not necessarily transverse,
 we prove that $\sim$ is an
equivalence relation on ${\mathcal P}_k(g, C_g(p_g))$, for every
$g\in {\mathcal U}$. (We could prove this for all periodic orbits in
$C_g(p_g)$. Nevertheless we are interested only in those with period
$\le k$.) We first prove this for $f$. Let $Q_1$ and $Q_2$ be two
periodic orbits of $f$ in ${\mathcal P}_k(f, C_f(p))$. If $Q_1$ and
$Q_2$ have the same index, by item (4) of Proposition \ref{Generic},
they form a (transverse) heteroclinic cycle. If $Q_1$ and $Q_2$ have
different indices, since $f$ is Kupka-Samle, either $W^s(Q_1, f)\cap
W^u(Q_2, f)=\emptyset$, or $W^s(Q_1, f)\cap W^u(Q_2, f)=\emptyset$,
hence $Q_1$ and $Q_2$ do not form a heteroclinic cycle. Thus
$Q_1\sim Q_2$ if and only if $Q_1$ and $Q_2$ have the same index,
hence $\sim$ is an equivalence relation on ${\mathcal P}_k(f,
C_f(p))$. But $h$ preserves heteroclinic cycles, hence $\sim$ is an
equivalence relation on ${\mathcal P}_k(g, C_g(p_g))$ too, and $h$
maps equivalence classes of ${\mathcal P}_k(f, C_f(p))$ to
equivalence classes of ${\mathcal P}_k(g, C_g(p_g))$.

Let
$$l_g=\max\{\sharp {\mathcal C}: {\mathcal C} \text{ is an equivalence class of} \
{\mathcal P}_k(g, C_g(p_g)) \text{ with respect to} \sim \},$$ where
$\sharp {\mathcal C}$ denotes the number of elements of ${\mathcal
C}$. Since $h$ preserves this number, $l_g$ is independent of $g\in
{\mathcal U}$, and will be denoted $l$ below. Now take an
equivalence class ${\mathcal C}$ of ${\mathcal P}_k(f, C_f(p))$ such
that $\sharp {\mathcal C}=l$. Let ${\mathcal
C}=\{Q_1,Q_2,\cdots,Q_l\}$. Note that ${\mathcal C}\neq {\mathcal
P}_k(f, C_f(p))$ as we have assumed for contradiction that periodic
orbits of ${\mathcal P}_k(f, C_f(p))$ do not have the same index.
There are two cases to consider:

\vskip 0.5 cm

{\noindent\bf Case 1.} ${\rm Orb}(p)\in {\mathcal C}$. In this case
${\rm Orb}(p)$  forms a (transverse) heteroclinic cycle with every
$Q_i\in {\mathcal C}$. Note that there is a periodic orbit $Q$
outside ${\mathcal C}$. Then ${\rm Orb}(p)$  does not form a
heteroclinic cycle with $Q$. Note that, by item (4) of Proposition
\ref{Generic}, either $W^s({\rm Orb}(p))\cap W^u(Q)\neq \emptyset$
or $W^u({\rm Orb}(p))\cap W^s(Q)\neq \emptyset$, according to which
paring of $W^s$ and $W^u$ for ${\rm Orb}(p)$ and $Q$ has adequate
dimensions. Thus there is only one paring  that needs be connected
and, as ${\rm Orb}(p)$ and $Q$ are in the same chain (actually
homoclinic) class $C_f(p)$, by the $C^1$ connecting lemma, they
indeed can be connected. Precisely, there is an arbitrarily small
$C^1$ perturbation $g$ of $f$ that creates a heteroclinic cycle of
$g$ associated with ${\rm Orb}(p_g)$ and $Q_g$. (See \cite{GW} for
some details of the perturbation.) That is, ${\rm Orb}(p_g)\sim
Q_g$. On the other hand, since the heteroclinic cycles formed by
${\rm Orb}(p)$ and $Q_i$, $i=1, ..., l$, respectively, are each
transverse, they survive if the perturbation is small enough. That
is, ${\rm Orb}(p_g)\sim (Q_i)_g$ for all $i=1, ..., l$. Of course
$Q_g, (Q_1)_g,$ ..., $(Q_l)_g$ are distinct if the perturbation is
small enough. Thus the equivalent class of ${\rm Orb}(p_g)$ contains
at least $l+1$ elements. This contradicts the definition of $l$.

\vskip 0.5 cm

{\noindent\bf Case 2.} ${\rm Orb}(p)\notin {\mathcal C}$. In this
case $Q_1$ forms a transverse heteroclinic cycle with every $Q_i\in
{\mathcal C}$. Note that ${\rm Orb}(p)$ is outside ${\mathcal C}$.
As discussed above, by the $C^1$ connecting lemma, there is an
arbitrarily small $C^1$ perturbation $g$ of $f$ that creates a
heteroclinic cycle associated with $(Q_1)_g$ and ${\rm Orb}(p_g)$.
That is, $(Q_1)_g\sim {\rm Orb}(p_g).$ If the perturbation is small
enough, the transverse heteroclinic cycle formed between $Q_1, ...,
Q_l$ survive. That is, $(Q_1)_g\sim (Q_i)_g$ for all $i=1, ..., l$.
Of course ${\rm Orb}(p_g), (Q_1)_g, ..., (Q_l)_g$ are distinct if
the perturbation is small enough. Thus the equivalence class of
${\rm Orb}(p_g)$, which is the equivalence class of  $(Q_1)_g$,
contains at least $l+1$ elements, contradicting the definition of
$l$. This proves that all periodic points in $C_f(p)$ have the same
index.

Thus, by item (4) of Proposition \ref{Generic}, every periodic point
$q\in C_f(p)$ is homoclinically related to $p$, proving Proposition
\ref{Proposition1}.

\end{proof}

The next proposition asserts that, in a structurally stable chain
class, eigenvalues of periodic orbits are uniformly and robustly
away from the unit circle.

\begin{Proposition}\label{Proposition22} Let $f\in\mathcal{R}$, and
let $p\in M$ be a hyperbolic periodic point of $f$. If $C_f(p)$ is
structural stable, then there are a constant $0<\lambda<1$ and a
neighborhood $\mathcal{U}$ of $f$ such that, for any
$g\in\mathcal{U}$ and any periodic point $q$ of $g$ homoclinically
related to $p_g$, the derivative $D_qg^{\pi(q)}$ has no eigenvalue
with modulus in $(\lambda,\lambda^{-1})$,  where ${\pi(q)}$ is the
period of $q$.
\end{Proposition}

\begin{proof} We prove  by contradiction. Suppose  there are a
diffeomorphism $g$ arbitrarily $C^1$ close to $f$ and a periodic
point $q\in C_g(p_g)$ homoclinically related to $p_g$ such that
$D_qg^{\pi(q)}$ has an eigenvalue arbitrarily close to $1$.  Denote
$\mu$ the eigenvalue which is closest to $1$, i.e.,
$|\log\mu|\leq|\log\mu'|$ for all eigenvalues $\mu'$ of
$D_qg^{\pi(q)}$. For explicitness we assume $|\mu|<1$. The case
$|\mu|>1$  can be treated similarly. Note that the notion of being
homiclinically related requires hyperbolicity of the periodic orbits
and transversality between the stable and unstable manifolds, hence
rules out the case $|\mu|= 1$. Also, since being homiclinically
related is a robust property, while keeping $q$ and $p_g$
homiclinically related, by taking an arbitrarily $C^1$ small
perturbation  we can assume that $\mu$ has multiplicity $1$ and $g$
is ``locally linear" near $g^iq$ in the sense that  there is $r>0$
such that
$$g|_{B_r(g^iq)}=\exp_{g^{i+1}q}\circ D_{g^iq}g\circ
\exp_{g^iq}^{-1}$$ for any $0\leq i<\pi(q)$. Let $E^c(q)\subset
T_qM$ be the eigenspace of $D_{q}g^{\pi(q)}$ associated to $\mu$. It
is a line if $\mu$ is real or a plane if $\mu$ is complex. In the
second case by taking another arbitrarily small  perturbation we can
assume that $D_{q}g^{\pi(q)}|_{E^c(q)}$ is a rational rotation of
the plane. For $\eta>0$, denote the ball in $E^c(q)$ of radius
$\eta$ about the origin to be $E^c(q, \eta)$.

We construct a perturbation $\tilde{g}$ of $g$. Let
$\alpha(x):[0,+\infty)\to[0,+\infty)$ be a bump function which
satisfies (1) $\alpha|_{[0,1/3]}=1$, (2)
$\alpha|_{[2/3,+\infty)}=0$, (3) $0<\alpha|_{(1/3,2/3)}<1$ and (4)
$0\leq \alpha'(x)<4$ for all $x\in[0,+\infty)$.  For a small
$\eta>0$, define a real function $\beta:T_qM\to \mathbb{R}$ by
$$\beta(v)=|\mu|^{-1}\alpha(|v|/\eta)+(1-\alpha(|v|/\eta)).$$ Thus
 $\beta(v)=|\mu|^{-1}$ for $|v|\leq\eta/3$,
$1<\beta(v)<|\mu|^{-1}$ for $\eta/3<|v|<2\eta/3$, and $\beta(v)=1$
for $|v|\geq2\eta/3$. We always assume $\eta$ much less than $r$.
Define a perturbation $\tilde{g}$ of $g$ to be
$$\tilde{g}(x)=\exp_{g(q)}(\beta(v)
\cdot D_{q}g(v)), \ \ v=\exp_{q}^{-1}(x)$$ for $x\in B(q, \eta)$,
and define $\tilde{g}(x)=g(x)$ for $x \notin B(q, \eta)$. Briefly,
in addition to the act of the tangent map $D_{q}g$, the perturbation
stretches vectors of length $\le \eta/3$ by a constant factor
$|\mu|^{-1}$, and stretches vectors of length between $\eta/3$ and
$2\eta/3$ by a variable factor $1<\beta(v)<|\mu|^{-1}$, and leaves
alone vectors of length $\ge \eta$. Then $\tilde{g}$ is $C^1$ close
to $g$ if $|\mu|$ is sufficiently close to 1. We take $\eta$ small
so that the $\pi(q)$ balls $B(g^i(q), \eta)$ are mutually disjoint.
To simplify notations we regard $p$ and $q$ below as fixed points.
We prove that $\exp_q(\overline{E^c(q, \eta/3))}$, which is an
interval if $\mu$ is real or a 2-disc if $\mu$ is complex, is
contained in $C_{\tilde{g}}(p_{\tilde{g}})$.

Since $q$ and $p_g$ are homoclinically related with respect to $g$,
there is $x^*\in W^s(q, g)\cap W^u(p_g, g)$ and $y^*\in W^u(q,
g)\cap W^s(p_g, g)$. Since $g$ is locally linear, we may assume
$x^*\in \exp_q(E^s_r)$ and $y^*\in\exp_q(E^u_r)$. Also, we may
assume that the positive orbit of $x^*$ and the negative orbit of
$y^*$ both remain in $B(q, r)$. If $\eta$ is small enough, the
negative orbits of $x^*$ and the positive orbit of $y^*$ will be
unchanged under the perturbation $\tilde{g}$. Hence
$\tilde{g}^{-n}(x^*)\to p_g$ and $\tilde{g}^n(y^*)\to p_g$ as
$n\to+\infty$. Now we consider the positive orbit of $x^*$ and the
negative orbit of $y^*$ under $\tilde{g}$.

Denote
 $${G}:T_qM\to T_qM$$
  $$G(v)=\beta(v)\cdot D_qg(v).$$
 Note that for  $v$ near the origin,
  $$G(v)=\exp_{g(q)}^{-1}\circ
\tilde{g}\circ\exp_{q}(v).
$$ We prove ${G}^{-n}(v)\to 0$ for every
$v\in E^u(q)$. Since $G$ differs from $D_qg$ by a factor  $\beta(v)$
only,  $G$ preserves $E^u(q)$. Moreover,
$${G}^{-n}(v)=\beta^{-1}({G}^{-n}v)\cdots\beta^{-1}({G}^{-1}v)D_qg^{-n}(v)$$
for any $n\geq 1$.  Since $\mu$ is the eigenvalue of $D_qg$ closest
to the unit circle, and since $1\leq\beta(v)\leq|\mu|^{-1}$, the
factor $\beta$ is strictly weaker than any of the eigenvalue of
$E^u(q)$. Then ${G}^{-n}(v)\to 0$. Taking $v=\exp^{-1}(y^*)$ then
gives
$$\tilde{g}^{-n}(y^*)\to q$$
 as $n\to+\infty$.

We prove ${G}^n(v)\to\overline{E^c(q, \eta/3)}$ for every $v\in
E^s(q)$. Note that  $E^s(q)$ splits into a direct sum
$E^{ss}(q)\oplus E^c(q)$ hence we may write $v=v^{ss}+v^c$. Write
$${G}^n(v)=v^{ss}_n+v^c_n.$$  Since $1\leq\beta(v)\leq|\mu|^{-1}$,
$$|v_n^{ss}|=|\beta({G}^{n-1}v)\cdots\beta(v)D_qg^{n}(v^{ss})|\leq
|\mu|^{-n}|D_qg^{n}(v^{ss})|. $$ Since $\mu$ is strictly weaker than
any of the eigenvalue of $E^{ss}(q)$, we get
$$\lim_{n\to+\infty}|v_n^{ss}|=0.$$

Then we check  $E^c$. First consider the case when  $\mu$ is real.
Then $E^c$ is a line. By the definition of $G$, the closed interval
$[-\eta/3, \ \eta/3]$ of $E^c$ consists of fixed points of $G$.
Since $D_qg(v)=\mu v$ and  since $\beta(v)<|\mu|^{-1}$ for any
$v\notin [-\eta/3, \ \eta/3]$,
 $$G(v)=\beta(v)\cdot T_qg(v)$$
is a strictly decreasing function on $E^c\setminus[-\eta/3, \
\eta/3]$. Thus every $v\in E^c\setminus[-\eta/3, \ \eta/3]$
approaches under $G$ along $E^c$ to one of the two end points of the
interval. Taking $v=\exp^{-1}(x^*)$ and writing $v=v^{ss}+v^c$ then
gives
$$\tilde{g}^{n}(x^*)\to \exp_q[-\eta/3, \
\eta/3]=\exp_q(\overline{E^c(\eta/3))}.$$ Clearly, any interval of
fixed points is a chain transitive set, meaning its points are
mutually chain equivalent. Thus the whole interval
$\exp_q(\overline{E^c(\eta/3))}$ is contained in
$C_{\tilde{g}}(p_{\tilde{g}})$.

Since $C_f(p)$ conjugates $C_{\tilde{g}}(p_{\tilde{g}})$, $C_f(p)$
also contains an interval of fixed points. This contradicts that $f$
is Kupka-Smale, proving Proposition \ref{Proposition22} in the case
when $\mu$ is real.

The case $\mu$ is complex  is proved similarly. In this case  $E^c$
is a plane $P$ and $\exp_q(\overline{E^c(\eta/3))}$ is a disc. Note
that we have assumed that $D_{q}g^{\pi(q)}$ is conjugate to a
rational rotation of $P$. Hence the disc
$\exp_q(\overline{E^c(\eta/3))}$ consists of periodic points of $G$
of the same period.  Thus the proof goes the same as the case when
$\mu$ is real. This proves  Proposition \ref{Proposition22}.
\end{proof}

Let $\Lambda\subset M$ be an invariant set of $f$. A splitting
$T_\Lambda M=E\oplus F$ is called {\it $(m,\lambda)$-dominated},
where $m\ge 1$ and $0<\lambda<1$, if
$$Df(E(x))=E(f(x)), \ Df(F(x))=F(f(x))$$
and
$$\|Df^m|_{E(x)}\|\cdot\|Df^{-m}|_{F(f^mx)}\|<\lambda$$
for every $x\in \Lambda$.   Since the constants $m\ge 1$ and
$0<\lambda<1$ are uniform, a dominated splitting over $\Lambda$
always extends to the closure $\overline \Lambda$. (See \cite{BDP}.)

A dominated splitting demands relative rates between the two
subbundles, rather than individual rates of each, which is what a
hyperbolic splitting  demands. A hyperbolic splitting is
automatically a dominated splitting, but not vise versa. Note that
if the dimensions of the summands are fixed, the dominated splitting
is unique. Thus, besides the interest of its own, a dominated
splitting often serves as a (unique) candidate for a possible
hyperbolic splitting. Indeed, if there is ever a hyperbolic
splitting, it must be this.

A fundamental tool that ensures the existence of a dominated
splitting is the perturbation theory of periodic linear co-cycles
developed by Liao \cite{Liao1} and Ma\~n\'e \cite{Man}.  Let
$\pi:E\to \Lambda$ be a finite dimensional vector bundle and
$f:\Lambda\to \Lambda$ be a homeomorphism. A continuous map $A:E\to
E$ is called a {\it linear co-cycle} (or {\it bundle isomorphism})
if $\pi\circ A=f\circ \pi$, and if $A$ restricted to every fiber is
a linear isomorphism.  (Note that the letter $\pi$ here denotes the
bundle projection, but not the period of a periodic point as used
above and below. This is the only place in this paper where $\pi$ is
used in this way.) The topology of $\Lambda$ is not relevant to our
aim here, and we assume that $\Lambda$ has the discrete topology. We
say $A$ is {\it bounded} if there is $N>0$ such that ${\rm
max}\{\|A(x)\|, \|A^{-1}(x)\|\}\le N$ for every $x\in \Lambda$,
where $A(x)$ denotes $A|_{E(x)}$. For two linear co-cycles $A$ and
$B$ over the same base map $f:\Lambda\to \Lambda$, define
$$d(A, B)={\rm sup}_{x\in \Lambda}\{\|A(x)-B(x)\|, \|A^{-1}(x)-B^{-1}(x)\|\}.$$

 A periodic point $p\in \Lambda$ of
$f$ is called {\it hyperbolic} with respect to $A$ if $A^{\pi(p)}$
have no eigenvalues of absolute value 1, where ${\pi(p)}$ is the
period of $p$. As usual, we denote the contracting and expanding
subspaces of $p$ to be $E^s(p)$ and $E^u(p)$. Then
$E(p)=E^s(p)\oplus E^u(p)$. If every point in $\Lambda$ is periodic
of $f$, then $A$ is called a {\it periodic  linear co-cycle}. A
bounded periodic linear co-cycle $A$ is called a {\it star system}
if there is $\epsilon>0$ such that any $B$ with $d(B, A)<\epsilon$
has no non-hyperbolic periodic orbits. (This notion corresponds to
that of diffeomorphisms on the manifold $M$ but, since perturbations
on manifolds are less restrictive, the star condition on manifolds
is stronger. In fact
 it implies Axiom A and no-cycle.) The next fundamental result of Liao and
Ma\~n\'e says that, if $A$ is a star system, then the individual
hyperbolic splittings $E^s(p)\oplus E^u(p)$ of $p\in \Lambda$, put
together, form a dominated splitting. It also gives some estimates
for  rates on periodic orbits.

\begin{theorem}\label{PLC} (\cite{Liao1}, \cite{Man})
Let  $A:E\to E$ be a bounded periodic linear co-cycle over
$f:\Lambda\to \Lambda$. If $A$ is a star system, then there is
$\epsilon>0$ and three constants $m>0$, $C>0$ and $0<\lambda<1$ such
that, for any linear co-cycle $B$ over  $f$ with $d(B, A)<\epsilon$,
and any periodic point $q$ of $B$, the following conditions are
satisfied:

$(1)$  $\|B^m|_{E^s(q)}\|\cdot\|B^{-m}|_{E^u(f^mq)}\|<\lambda$.

$(2)$ Let $k=[\pi(q)/m]$, then
$$\prod_{i=0}^{k-1}\|B^m|_{E^s(f^{im}(q))}\|<C\lambda^k,$$
$$\prod_{i=0}^{k-1}\|B^{-m}|_{E^u(f^{-im}(q))}\|<C\lambda^k.$$
\end{theorem}

The two inequalities in Item 2 are usually referred to as
``uniformly contracting (expanding) at the periods" for periodic
orbits. We remark that Liao and  Ma\~n\'e did not use the term of
linear co-cycles. Liao worked  (for flows) on tangent bundles of
manifolds, and  Ma\~n\'e worked on
 periodic sequences of linear isomorphisms.

Via Franks' lemma \cite{Fra}, Theorem \ref{PLC} applies to the
manifold $M$ and ensures a dominated splitting for certain set of
periodic orbits of $f$. The classical application is the one in the
proof of the stability conjecture by Ma\~n\'e \cite{Man2}. We do not
state the Franks lemma as we will need a refined Franks lemma that,
briefly, preserves intersections of stable and unstable manifolds,
because we have to always stay inside the chain class. This is the
result of Gourmelon \cite{Gou}. We take a simple form of his result
that is enough to our purpose:

\begin{Proposition}\label{Gourmelon} (\cite{Gou}) \ Let $f$ be a diffeomorphism of $M$.
For any $C^1$ neighborhood ${\mathcal U}$ of $f$, there is
$\epsilon>0$ such that, for any pair of hyperbolic periodic points
$p, q\in M$ of $f$ that are homoclinically related, any neighborhood
$U$ of  ${\rm Orb}(q)$ in $M$ not touching ${\rm Orb}(p)$, and any
continuous path of linear isomorphisms $A_{k, t}: T_{f^k q}M\to
T_{f^{k+1} q}M$ that satisfies the following three assumptions:

$(1)$  $A_{k, 0}=D_{f^kq}f$ for all $0\leq k<\pi(q)$,

$(2)$ $\|A_{k, t}- D_{f^k(q)}f\|<\epsilon$ for all $0\leq k<\pi(q)$
and any $t\in[0, 1]$,

$(3)$ $A_{\pi(q)-1, t}\circ A_{\pi(q)-2, t}\circ\cdots\circ A_{0,
t}$ has no eigenvalue on the unit circle for all $t\in[0, 1]$,

there exist a perturbation $g\in\mathcal{U}$ with the following
three properties:

$(A)$  $g=f$ on  $(M\backslash U)\cup{\rm Orb}(q)$,

$(B)$  $D_{f^kq}g=A_{k,1}$ for all $0\leq k<\pi(q)$,

$(C)$  $p$ and $q$ are homoclinically related with respect to $g$.

\end{Proposition}

\begin{Proposition}\label{Proposition23} Let
$f\in\mathcal{R}$, and let $p\in M$ be a hyperbolic periodic point
of $f$. If $C_f(p)$ is structurally stable then there are three
constants $m>0$, $C>0$  and $0<\lambda<1$ such that, for any
periodic point $q$ of $f$ that is homoclinically related to $p$, the
following conditions are satisfied:

$(1)$  $\|Df^m|_{E^s(q)}\|\cdot\|Df^{-m}|_{E^u(f^mq)}\|<\lambda$.

$(2)$  Let $k=[\pi(q)/m]$, then
$$\prod_{i=0}^{k-1}\|Df^m|_{E^s(f^{im}(q))}\|<C\lambda^k,$$
$$\prod_{i=0}^{k-1}\|Df^{-m}|_{E^u(f^{-im}(q))}\|<C\lambda^k.$$
\end{Proposition}

\begin{proof}
Let ${\mathcal U}$ and $0<\lambda<1$ be given in Proposition
\ref{Proposition22}. For this ${\mathcal U}$, let $\epsilon>0$ be
the number given in Proposition \ref{Gourmelon}.

Let $\Lambda$ be the union of periodic orbits of $C_f(p)$. By
Proposition \ref{Proposition1}, every $q\in \Lambda$ is
homoclinically related to $p$. The tangent map $Df:T_\Lambda M\to
T_\Lambda M$ acts as a periodic linear co-cycle over $f$. We verify
that it is a star system in the sense of linear co-cycles.

Suppose for the contrary there is a linear co-cycle $A:T_\Lambda
M\to T_\Lambda M$ arbitrarily close to $Df$ that has a periodic
orbit ${\rm Orb}(q)$ of $f$ which is non-hyperbolic with respect to
$A$. We join $A$ with $Df$ by a path $A_t$ with $A_0=Df$ and
$A_1=A$. Since $A$ can be arbitrarily close to $Df$, we may assume
$A_t|_{{\rm Orb}(q)}$ satisfies assumption (2) of Proposition
\ref{Gourmelon}, for every $t\in [0, 1]$. Let $s\in (0, 1]$ be the
first parameter that makes $q$ non-hyperbolic, namely, $q$ is
non-hyperbolic with respect to $A_s$, but is hyperbolic with respect
to $A_t$, for every $t\in [0, s)$. Take $s'$ slightly less than $s$
so that one of the eigenvalues $\mu$ of $A_{\pi(q)-1, s'}\circ
A_{\pi(q)-2, s'}\circ\cdots\circ A_{0, s'}$ (in absolute value) is
within $(\lambda, \lambda^{-1})$. Then the path $A_t$, $t\in [0,
s']$, satisfies the three assumptions of Proposition
\ref{Gourmelon}, hence there is $g\in {\mathcal U}$  that preserves
${\rm Orb}(q)$ and ${\rm Orb}(p)$ and realizes $A_{s'}|_{{\rm
Orb}(q)}$ to be $Dg|_{{\rm Orb}(q)}$, such that $p$ and $q$ are
homoclinically related with respect to $g$. Such a weak eigenvalue
$\mu$ contradicts Proposition \ref{Proposition22}. This verifies
that $Df:T_\Lambda M\to T_\Lambda M$ is a star periodic linear
co-cycle over $f$. Thus Proposition \ref{Proposition23} follows from
Theorem \ref{PLC}.
\end{proof}

\section{Minimally non-contracting sets}
The following  result is well known as Pliss lemma.
\begin{Proposition}\label{Pliss} {\rm (Pliss)} {\rm \cite{Pli}}
Let $K>0$ and $\gamma_1<\gamma_2$ be given. There is $c>0$ such that
for any sequence of real numbers  $a_0,..., a_{n-1}$  with $|a_i|\le
K$, if
$$\frac{1}{n}\sum_{i=0}^{n-1}a_i<\gamma_1,$$ then there are
$0\leq n_1<...<n_j\leq n-1$ such that
$$\frac{1}{k}\sum_{i=n_m}^{n_m+k-1}a_i<\gamma_2$$ for all $1\le m\le j$ and $1\leq
k\leq n-n_m$. Furthermore, $j\ge cn$.
\end{Proposition}

Briefly, Pliss lemma says that if a finite sequence $a_0, ...,a_n$
has total (from $0$ to $n$) average
 less than $\gamma_1$, then there are proportionally many intermediate times $n_m$
 such that the averages from $n_m$ to all its successors $n_m+k$ are less than $\gamma_2$.
This elementary lemma will be frequently used below.

Let $\Lambda\subset M$ be a compact invariant set of $f$ and $E$ be
a continuous subbundle of $T_\Lambda M$, invariant under $Df$. For
$x\in \Lambda$, denote $$\phi(x)=\log\|Df|_{E(x)}\|.$$ Thus $\phi$
is a real function on $\Lambda$ about exponential rates of $Df$ on
$E$ under positive iterates. Here is a corollary of Pliss lemma:

\begin{lemma}
$(1)$ If there is $x\in \Lambda$ with
$\liminf_{n\to+\infty}\frac{1}{n}\sum_{i=0}^{n-1}\phi(f^{i}x)<0$,
then there is $y\in \Lambda$ with
$\limsup_{n\to+\infty}\frac{1}{n}\sum_{i=0}^{n-1}\phi(f^{i}y)<0$.

 $(2)$ If there is $x\in \Lambda$ with
$\limsup_{n\to+\infty}\frac{1}{n}\sum_{i=0}^{n-1}\phi(f^{i}x)>0$,
then there is $y\in \Lambda$ with
$\liminf_{n\to+\infty}\frac{1}{n}\sum_{i=0}^{n-1}\phi(f^{i}y)>0$.
\end{lemma}
\begin{proof}
If there is $x\in \Lambda$ such that
$$s=\liminf_{n\to+\infty}\frac{1}{n}\sum_{i=0}^{n-1}\phi(f^{i}x)<0,$$
then there are positive integers $n_j\to +\infty$ such that
$$\sum_{i=0}^{n_j-1}\phi(f^{i}x)<n_j(s+\varepsilon)$$ for a
small $\varepsilon\in(0, -s/2)$. By Pliss lemma, there is $m_j$ for
every $j$ such that $n_j-m_j\to+\infty$ and
$$\frac{1}{k}\sum_{i=m_j}^{m_j+k-1}\phi(f^{i}x)<s+2\varepsilon$$
for any $k=1,\cdots,n_j-m_j$. By taking a subsequence, we may assume
$f^{m_j}x\to y\in \Lambda$. One can verify that
$$\limsup_{n\to+\infty}\frac{1}{n}\sum_{i=0}^{n-1}\phi(f^{i}y)\leq s+2\varepsilon<0.$$
Item (2) can be proven similarly. This proves Lemma 3.1.
\end{proof}

We also need the following result known as Liao's selecting lemma,
see \cite{Liao}. There is an exhibition for this lemma in
\cite{Wen}.
\begin{Proposition}\label{Liao} {\rm (Liao)}
 Let $\Lambda$ be a compact invariant set of $f$ with $(m,
\lambda)$-dominated splitting $E\oplus F$ with ${\rm dim}(E)=I$,
$1\le I\le d-1$. Assume

$(1)$ There is a point $b\in\Lambda$ satisfying
$$\prod_{i=0}^{n-1}\|Df^m|_{E(f^{im}b)}\|\ge 1$$
for all $n\ge 1.$

$(2)$  \ There are $\lambda_1$ and $\lambda_2$ with
$\lambda<\lambda_1<\lambda_2<1$ such that for any $x\in\Lambda$
satisfying
$$\prod_{i=0}^{n-1}\|Df^m|_{E(f^{im}x)}\|\ge \lambda_2^n$$
for all $n\ge 1,$  $\omega(x)$ contains a point $c\in\Lambda$
satisfying
$$\prod_{i=0}^{n-1}\|Df^m|_{E(f^{im}c)}\|\le \lambda_1^{n}$$
for all $n\ge 1$.

Then for any $\lambda_3$ and $\lambda_4$ with
$\lambda_2<\lambda_3<\lambda_4<1$, there is a sequence of hyperbolic
periodic point $q_n$ of $f$ of index $I$ such that

$(A)$ ${\rm Orb}(q_n)$ converge to a subset of $\Lambda$ in the
Hausdorff metric;

$(B)$  ${\rm Orb}(q_n)$ are mutually homoclinically related;

$(C)$ the periods $\pi(q_n)$ are multiples of $m$ such that
$$\prod_{i=0}^{k-1}\|Df^m|_{E^s(f^{im}q_n)}\|\le \lambda_4^{k},$$
$$\prod_{i=k-1}^{\pi(q_n)/m-1}\|Df^m|_{E^s(f^{im}q_n)}\|\ge
\lambda_3^{\pi(q_n)/m-k+1},$$ for all $k=1,\cdots,\pi(q_n)/m$. Here
$E^s$ denotes the stable subbundle of ${\rm Orb}(q_n)$.

 Similar assertions for $F$ hold respecting
$f^{-1}$.
\end{Proposition}

Thus, by taking $\lambda_3$ close to 1, the $E^s$-rates  at the
periods for ${\rm Orb}(q_n)$ could be arbitrarily weak. Note that
(B) was not included in the statement of the selecting lemma in
\cite{Wen}. For convenience of application we have added (B) here.
It is a consequence of (C). In fact, since ${\rm Orb}(q_n)$ is
periodic, applying Pliss lemma to this special case, one can find a
point $x_n\in {\rm Orb}(q_n)$ such that the $E^s$ rates of $x_n$
from $0$ to $\infty$ are all less than a slightly larger
$\lambda'_4$. This guarantees certain uniform size of
$W^s_{loc}(x_n)$. Likewise for $W^u_{loc}(x_n)$. Taking subsequences
we may assume $x_n$ converge to a point of $\Lambda$ hence, for $n$
large, $x_n$ are mutually homoclinically related. This gives (B).

Let $E$ be a continuous subbundle of $T_\Lambda M$.  As usual, $E$
is called {\it contracting} if there are $m\ge 1$ and $0<\lambda<1$
such that
$$\|Df^m|_{E(x)}\|\le \lambda$$
for any $x\in \Lambda$. In the spirit of Liao \cite{Liao} we call a
compact invariant set $K\subset \Lambda$ of $f$ {\it minimally
non-contracting of} $E$ if $E|_K$ is not contracting but $E|_{K'}$
is contracting for any compact invariant proper subset $K'\subset
K$. By Zorn's Lemma, every non-contracting set $\Lambda$ of $E$
contains a minimally non-contracting subset of $E$.

Let $f\in\mathcal{R}$, and let $C_f(p)$ be a structurally stable
chain class of $f$. By Proposition \ref{Generic}, $C_f(p)=H(p, f)$,
hence periodic points are dense in $C_f(p)$. By Proposition
\ref{Proposition1}, every periodic point in $C_f(p)$ is
homoclinically related to $p$.  Then the ($m, \lambda$)-dominated
splittings on these periodic orbits, obtained by Proposition
\ref{Proposition23}, extend to a dominated splitting
$$T_{C_f(p)}M=E\oplus F$$
on the whole set $C_f(p)$ with the same constants $m\ge 1$ and
$0<\lambda<1$. Note that $\dim E={\rm Ind}(p)$. Restricted to
periodic points $q\in C_f(p)$, one has $E(q)=E^s(q)$ and
$F(q)=E^u(q)$. We will work with this splitting throughout below and
eventually prove  it is hyperbolic, assuming ${\rm Ind}(p)$ is 1 or
$\dim M-1$.

Now assume ${\rm Ind}(p)=1$ and hence $\dim E=1$. The case ${\rm
Ind}(p)=\dim M-1$ can be treated similarly. We prove that, in this
case, any minimally non-contracting sets of $E$ must be partially
hyperbolic. Recall a dominated splitting $E\oplus F$ is called {\it
partially hyperbolic} if either $E$ is contracting, or $F$ is
expanding.

\begin{Proposition}\label{Proposition31}
Let $f\in\mathcal{R}$, and let $C_f(p)$ be a structurally stable
chain class of $f$. Let $T_{C_f(p)}M=E\oplus F$ be the dominated
splitting as above, and assume $\dim E=1$. Let $\Lambda\subset
C_f(p)$ be a minimally non-contracting set of $E$. Then $\Lambda$ is
partially hyperbolic. Indeed, writing $T_\Lambda M=E^c\oplus E^u$,
where $E^c=E|_\Lambda$ and $E^u=F|_\Lambda$, then $E^u$ is
expanding. Moreover,
$$\lim_{n\to+\infty}\frac{1}{n}\sum_{i=0}^{n-1}\log\|Df^m|_{E^c(f^{im}x)}\|=0$$
for all $x\in \Lambda$, where $m$ is the constant given in
Proposition \ref{Proposition23}.
\end{Proposition}
\begin{proof}
It suffices to  prove the limit equality only, as it directly
implies that $F|_\Lambda$ is expanding, by domination. We prove by
contradiction. Abbreviate
$$\phi(x)=\log\|Df^m|_{E(x)}\|$$ for $x\in \Lambda$. Here $E=E^c$. Suppose  there
is $y\in \Lambda$ such that
$$\limsup_{n\to+\infty}\frac{1}{n}\sum_{i=0}^{n-1}\phi(f^{im}y)>0.$$
By (a variant use of) Pliss lemma, there are $\lambda_1\in(0,1)$ and
a sequence of positive integers $n_1<n_2<\cdots$ such that
$$\frac{1}{k}\sum_{i=n_j-k}^{n_j-1}\phi(f^{mi}y)>-\log\lambda_1,$$
or, what is the same,
$$\prod_{i=n_j-k}^{n_j-1}\|Df^{m}|_{E(f^{mi}(y))}\|>\lambda_1^{-k},$$
 for any $j\geq 1$ and any $1\leq k\leq n_j$.
Taking inverse then gives
$$\prod_{i=n_j-k+1}^{n_j}\|Df^{-m}|_{E(f^{mi}(y))}\|<\lambda_1^{k},$$
 for any $j\geq 1$ and any $1\leq k\leq n_j$.  Note that this is the place where we use
 the assumption ${\rm dim}E=1$ (otherwise the inequality would be about mininorm instead of norm).
 Since $E\oplus F$ is a dominated splitting,
$$\prod_{i=n_j-k+1}^{n_j}\|Df^{-m}|_{F(f^{mi}(y))}\|<\lambda_1^{k},$$
 for any $j\geq 1$ and any $1\leq k\leq n_j$. Since the angles between
 $E(x)$ and $F(x)$ have a positive minimum for all $x\in C_f(p)$,
 switching to an equivalent norm if necessary, we may assume
$$\prod_{i=n_j-k+1}^{n_j}\|Df^{-m}(f^{mi}(y))\|<\lambda_1^{k},$$
 for any $j\geq 1$ and any $1\leq k\leq n_j$. Here $\|Df^{-m}(x)\|$
 denotes (as usual)  the norm of $Df^{-m}$ on the whole tangent space $T_x M$.
Briefly, $n_j$ are ``hyperbolic times" (or more precisely,
``contracting times") of $Df^{-m}$. Take a limit point $z$ of
$\{f^{n_jm}(y)\}_{j=1}^\infty$, it is standard to check that ${\rm
Orb}(z)$ is a periodic source of $f$. But $z\in \Lambda\subset
C_f(p)$, contradicting that any chain class can not contain a
periodic source unless the class reduces to this source. This proves
$$\limsup_{n\to+\infty}\frac{1}{n}\sum_{i=0}^{n-1}\phi(f^{im}x)\leq 0$$
for all $x\in \Lambda$.

Next suppose there is $y\in \Lambda$ such that
$$\liminf_{n\to+\infty}\frac{1}{n}\sum_{i=0}^{n-1}\phi(f^{im}y)<0.$$
By Lemma 3.1, the set $$S=\{s<0: \text{there is} ~ x\in \Lambda ~
\text{with} ~
\limsup_{n\to+\infty}\frac{1}{n}\sum_{i=0}^{n-1}\phi(f^{im}x)=s\}$$
is nonempty. There are two possibilities: $\sup S=0$ or $\sup S<0$.

If  $\sup S=0$, then there is $z\in \Lambda$ such that
$$\log\lambda<\limsup_{n\to+\infty}\frac{1}{n}\sum_{i=0}^{n-1}\phi(f^{im}z)<0.$$
Applying Theorem 2 of \cite{WD}, we obtain a hyperbolic periodic
point $q$ of $f$ such that $z\in H(q, f)$ and
$$\prod_{i=0}^{\pi(q)-1}\|Df^m|_{E^s(f^{im}(q))}\|>\lambda^{\pi(q)}.$$
Moreover, we have $\Lambda\cap H(q, f)\not=\emptyset$ and hence
$H(q, f)=C_f(p)$. Note that in the homoclinic classes, we can choose
a periodic point with arbitrarily large period such that the above
inequality is satisfied, this contradicts  Proposition
\ref{Proposition23}.

If  $\sup S<0$, we prove that $\Lambda$ satisfies the two
assumptions of Liao's selecting Lemma. Note that, since $E|_\Lambda$
is not contracting, there is a point $b\in \Lambda$ such that
$$\prod_{i=1}^{k-1}\|Df|_{E(f^{im}b)}\|\geq 1$$ for any $k\geq 1$. Thus the first assumption
of Liao's lemma is verified.

Now take $\xi_1$ and $\xi_2$ with
$$\max\{\lambda,e^{\sup S}\}<\xi_1<\xi_2<1.$$ To verify the second assumption of Liao's selecting lemma, let $x\in
\Lambda$ be a point with
$$\prod_{i=0}^{n-1}\|Df^m|_{E(f^{im}x)}\|\geq \xi_2^n$$ for all $n\ge
1.$ We verify that $\omega(x)$ contains a point $c$ with
$$\frac{1}{n}\sum_{i=0}^{n-1}\phi(f^{im}c)\leq \log\xi_1$$
for all $n\ge 1$. If $\omega(x)=\Lambda$,  there is of course a
point $y\in \omega(x)$ such that
$$\limsup_{n\to+\infty}\frac{1}{n}\sum_{i=0}^{n-1}\phi(f^{im}y)\leq\sup
S.$$  If $\omega(x)\not=\Lambda$, $E|_{\omega(x)}$ must be
contracting as $\Lambda$ is minimally non-contracting of $E$. Hence
$$\limsup_{n\to+\infty}\frac{1}{n}\sum_{i=0}^{n-1}\phi(f^{im}y)<0$$
for every $y\in \omega(x)$. Since $\sup S<0$, by the definition of
$S$,
$$\limsup_{n\to+\infty}\frac{1}{n}\sum_{i=0}^{n-1}\phi(f^{im}y)\leq\sup
S$$ for every $y\in\omega(x)$. Hence in both cases there is a point
$y\in\omega(x)$ such that
$$\limsup_{n\to+\infty}\frac{1}{n}\sum_{i=0}^{n-1}\phi(f^{im}y)\leq\sup
S.$$ Then, by  Pliss lemma, there is a point $c\in \omega(x)$ such
that
$$\frac{1}{n}\sum_{i=0}^{n-1}\phi(f^{im}c)\leq \log\xi_1$$
for all $n\ge 1$. This verifies the second assumption of Liao's
selecting lemma. Thus, by conclusion (C) of the lemma, there is  a
hyperbolic periodic point $q$ of $f$ such that
$$\prod_{i=0}^{\pi(q)-1}\|Df^m|_{E^s(f^{im}(q))}\|>\lambda^{\pi(q)}.$$
Moreover, by conclusion (A) and (B) of the lemma, we may assume
$\Lambda\cap H(q, f)\not=\emptyset$ and hence $H(q, f)=C_f(p)$. This
also contradicts Proposition \ref{Proposition23}, and proves
$$\liminf_{n\to+\infty}\frac{1}{n}\sum_{i=0}^{n-1}\phi(f^{im}(x))\geq
0$$ for all $x\in \Lambda$. Thus
$$\lim_{n\to+\infty}\frac{1}{n}\sum_{i=0}^{n-1}\log\|Df^m|_{E(f^{im}(x))}\|=0$$
for all $x\in \Lambda$. This proves Proposition \ref{Proposition31}.
\end{proof}

\section {A double  existence of periodic orbits}
The next result  asserts  a ``double" existence of a periodic orbit
near a minimal set $K$, i.e., the existence of a periodic orbit
which is, simultaneously, near $K$ and inside the chain class of
$K$.

\begin{theorem}\label{mainprop} Let $f\in\mathcal{R}$, and let $K$ be a
non-trivial minimal set with a partially hyperbolic splitting
$T_KM=E^c\oplus E^u$ such that $E^c$ is 1-dimensional and $E^u$ is
expanding. Then for any neighborhood $U$ of $K$ in $M$, there exists
a periodic orbit ${O}\subset U$ such that ${O}$ is in the chain
class of $K$.

\end{theorem}
To prove Theorem \ref{mainprop} we use the $\delta$-interval
argument taken from Pujals-Sambarino \cite{PS2}, combined with ideas
from the more recent central model theory of Crovisier \cite {Cro2}.
\begin{proof}
Since $E^u$ is expanding, the stable manifolds theorem guarantees a
family of local unstable manifolds $W_{loc}^u(x)$ tangent to $E^u$
at every $x\in K$. There is a neighborhood $U_0$ of $K$ such that
$W_{loc}^u(x)$ is defined for every $x\in \bigcap_{n\leq 0}
f^n(U_0)$. For any point $x\in\bigcap_{n\leq 0} f^n(U_0)$ and $y\in
W_{loc}^u(x)$, the distance $d(f^{-n}(x),f^{-n}(y))$  converges
exponentially to $0$.

There is a family of central manifolds $W_{loc}^c(x)$ tangent to
$E^c$ at every $x\in K$. The definition is more delicate.  Indeed,
by Hirsch-Pugh-Shub \cite{HPS} (also see \cite{PS2}), there is (not
uniquely) a continuous map
$$\phi^c: \bigcap_{n\geq 0} f^n(U_0)\to Emb([-1,1],M)$$ that gives a
family of central manifolds
$$W_{loc}^c(x)=\phi^c(x)[-1,1]$$
such that
$$\phi^c(x)(0)=x,$$
$$T_xW_{loc}^c(x)=E^c(x).$$ The family is invariant in the
sense that, for any $\epsilon>0$, there exist $\epsilon'>0$ such
that
$$f(W^c_{\epsilon'}(x))\subset W^c_{\epsilon}(fx),$$ where
$W_{\varepsilon}^c(x)=\phi^c(x)[-\varepsilon,\varepsilon]$. The
embedded disk $W_{loc}^c(x)$ is called the {\it center manifold}
through $x$. We fix such a map $\phi^c$ in this section, and hence
fix a family $W_{loc}^c(x)$ of central manifolds for $x\in K$. There
exists a neighborhood $U_0$ of $K$ such that the center manifolds
$W_{loc}^c(x)$ are defined for every $x\in \bigcap_{n\geq 0}
f^n(U_0)$. We often call a center manifold or a subinterval of it as
a {\it central segment}.

For any central segment $I=\phi^c(x)[0,\epsilon]$ or
$I=\phi^c(x)[-\epsilon,0]$, we say $I$ has length $\epsilon$,
denoted by $l(I)=\epsilon$. We will say  a central segment $I$ is
{\it based on} $x\in K$ if $x$ is an end point of $I$. As
 in Pujals-Sambarino \cite{PS2}, a non-trivial
central segment $I$ based on $x\in K$ is called a $\delta$-$E^c$
{\it segment} if
$$l(f^{-n}(I))\leq\delta$$ for any $n\geq 0$. If $I$ is a $\delta$-$E^c$
segment, so is $f^{-n}(I)$ for any $n\geq 0$. If
$0<\delta_1<\delta_2$, then a $\delta_1$-$E^c$ segment is
automatically  a $\delta_2$-$E^c$ segment. Note that if $I$ is a
$\delta$-$E^c$ segment, some bigger $I'\supset I$ could also be a
$\delta$-$E^c$ segment. But one can always extend $I$ to a biggest
$\delta$-$E^c$ segment.

\vskip 0.3 cm

{\noindent\bf Case 1.}  For any $\delta>0$, there is a
$\delta$-$E^c$ segment $I$ based on some point of  $K$.

This condition is weaker than to say there is a Lyapunov stable
point $x\in K$ which means,  for any $\delta>0$, there is a
$\delta$-$E^c$ segment $I$ based on the same $x\in K$. (Here we
consider negative iterates, and consider a one-side neighborhood $I$
of $x$ only.)

Let $\delta>0$ be arbitrarily given. We fix $\delta$ till the end of
Case 1. Let $I$ be a $\delta$-$E^c$ segment based on some point
$z\in K$. Denote $I_{f^{-n}z}$ the biggest $\delta$-$E^c$ segment
containing $f^{-n}(I)$. Note that $$f^{-1}
(I_{f^{-n}z})\not=I_{f^{-n-1}z}$$ in general. Nevertheless
$$f^{-k}(I_{f^{-n}z})\subset I_{f^{-n-k}z}$$ for all $k\geq 0.$

\vskip 0.3 cm {\noindent\bf Claim 1.} There is a subsequence
$n_k\to+\infty$ such that $l(I_{f^{-n_k}z})\to 0$ as $k\to\infty$.

In fact,  suppose for contradiction
$$\inf\{l(I_{f^{-n}z}):n\geq 0\}>0.$$  Since $K$ is minimal,
there exist positive integers $m_1>m_2$ such that
$$W_{loc}^u(f^{-m_1}z)\cap I_{f^{-m_2}z}\neq\emptyset.$$
Note that $I_{f^{-m_2}z}$ is also a $\delta-E^c$ segment and
$f^{-m_2}I\subset I_{f^{-m_2}z}$. By Theorem 3.1 of \cite{PS3}, the
$\alpha$-limit set
$$\alpha(I_{f^{-m_2}z})=\bigcup_{x\in I_{f^{-m_2}z}}\alpha(x)$$ falls into one of the following
four cases:
\begin{enumerate}
\item $\alpha(I_{f^{-m_2}z})\subset {C}$ where ${C}$ is a periodic simple closed curve normally
contracting for $f^{-m}$ where $m$ is the period of ${C}$ such that
$f^{-m}|_C$ has no periodic points;
\item There exists a normally attracting periodic arc $J$ such that $I_{f^{-m_2}z}\subset W^u(J)$ and
$f^k$ restricted to $J$ ($k$ being the period of $J$) is the
identity map on $J$;
\item $\alpha(I_{f^{-m_2}z})\subset {\rm Per}(f)$. Moreover, one of the periodic points is either a semi-expanding
periodic point or an expanding one.
\item $I_{f^{-m_2}z}$ is wandering.

\end{enumerate}

Since, as mentioned above, there exist positive integers $m_1>m_2$
such that
$$W_{loc}^u(f^{-m_1}z)\cap I_{f^{-m_2}z}\neq\emptyset,$$
Case (4) is ruled out. We verify that each of the other three cases
leads to a contradiction. In Case 1 ${\rm Orb}(C)$ is normally
expanding hence locally maximal.
 By Item (5) of Proposition \ref
{Generic}, there is a periodic orbit $P$ of $f$ in any small
neighborhood of ${\rm Orb}(C)$. Then $P\subset {\rm Orb}(C)$. Thus
$f^m|_C$ has periodic points, ruling out Case 1. Case 2 is directly
ruled out because $f$ is Kupka-Smale. In Case 3 $\alpha(z)$ is a
periodic orbit. This contradicts that $K$ is a non-trivial minimal
set because $z\in K$. This proves Claim 1.

\vskip 0.3 cm

{\noindent\bf Claim 2.}  There is a $\delta$-$E^c$ segment $J$ based
on some point $a\in K$ such that $J$ is contained in the chain class
of $K$.

Let $n_k$ be the sequence given in Claim 1. Take $I_k$ to be an
$E^c$-segment based on $f^{-n_k}z$ that is slightly larger than
$I_{f^{-n_k}z}$. Since $I_{f^{-n_k}z}$ is a biggest $\delta$-$E^c$
segment and $I_k$ is strictly larger, one of the (negative) iterates
of $I_k$ has length near $\delta$. Since $l(I_{f^{-n_k}z})\to 0$ by
Claim 1, we may assume $l(I_k)\to 0$ hence there is an integer $t_k$
such that $l(f^{-t_k}I_k)\geq\delta/2$ but $l(f^{-i}I_k)<\delta/2$
for all $0\leq i<t_k$. By taking subsequences, we may assume
$f^{-t_k}I_k$ accumulate to a non-trivial central segment $I'$ based
on a point $b\in K$. Since
 $l(I_k)\to 0$, the segment $I'$
goes into $K$ in the sense of chains, i.e., for any point $y\in I'$
and any $\varepsilon>0$, there is an $\varepsilon$-pseudo orbit
$y=x_0,x_1,\cdots, x_n$ such that $x_n\in K$.

Let $I'_0=I'$, and
$$I'_n=f^{-1}(I'_{n-1})\cap W_{loc}^c(f^{-n}(b)).$$
Being preimages of $I'$, $I'_n$ also goes into $K$ in the sense of
chains.

We search for a non-trivial $E^c$-segment that not only goes into
$K$, but also ``comes from" $K$, in the sense of chains. If
$$\inf\{l(I'_n):n\geq 0\}=0,$$ then one can take positive integers
$n_k$ and $m_k$ such that $l(I'_{n_k})\geq\delta/2$,
$l(f^{-m_k}(I'_{n_k}))\to 0$, and $l(f^{-i}(I'_{n_k}))\leq \delta$
for all $0\leq i\leq m_k$. Let $J$ be an accumulation segment of
$I'_{n_k}$. It is easy to see $J$ comes from and goes into $K$ in
the sense of chains, i.e.,  for any $y\in J$ and any
$\varepsilon>0$, there exists an $\varepsilon$-pseudo orbit starting
from $K$ and ending at $y$, and an $\varepsilon$-pseudo orbit
starting from $y$ and ending at $K$. In other words, $J$ is
contained in the chain class of $K$.

On the other hand, if
$$\inf\{l(I'_n):n\geq 0\}>0,$$
then by the minimality of $K$,  we can find a subsequence of $I'_n$
accumulating to some central segment $J$  such that $J\cap I$ is a
non-trivial interval, where $I$ is the $\delta$-$E^c$ segment based
on $z\in K$ given at the beginning of the proof for Case 1. (More
detailed discussion on the orientations of the central models of
Crovisier \cite {Cro2} ensures that $J$ can be chosen so that $J$
and $I$ are on the same side of $z$ so that $J\cap I$ is not a
single point $z$.) Then $J$ still goes into $K$ in the sense of
chains. But $l(f^{-n_k}I)\to 0$, hence $I$ comes from $K$ in the
sense of chains. Thus $J\cap I$ comes from and goes into $K$ in the
sense of chains. In other words, $J\cap I$ is contained in the chain
class of $K$.  This proves Claim 2.

Now let  $J$ be a $\delta$-$E^c$ segment based on $a\in K$ that
meets the requirement of Claim 2.  Note that $\bigcup_{y\in
J}W_{\delta}^u(y)$ forms a neighborhood of $J$ in $M$. Since
$f\in\mathcal{R}$, and since $J$ is contained in the chain class of
$K$, by Item 5 of Proposition \ref{Generic}, there is a periodic
point $p\in\bigcup_{y\in {\rm int}(J)}W_{\delta}^u(y)$. Thus there
exists a point $y_0\in {\rm int}(J)$ such that
$$d(f^{-n}y_0,f^{-n}(p))\to 0, ~ n\to+\infty,$$ hence $p$ is
contained in the chain class of $K$. Moreover,
$$d(f^{-n}y_0,f^{-n}(p))\le \delta, ~ d(f^{-n}y_0,f^{-n}(a))\le
\delta$$ for all $n\ge 0$. Hence ${\rm Orb}(p)$ is contained in the
$2\delta$-neighborhood of $K$. Since $\delta$ can be taken arbitrary
from the very beginning of the proof for Case 1, this proves Theorem
\ref{mainprop} in Case 1.

\vskip 0.3cm
 {\noindent\bf Case 2.} For some $\delta>0$, there is no $\delta$-$E^c$
segment based on a point of $K$.

This condition is sometimes referred to as ``sensitive dependence on
initial conditions" (see \cite {B}) which, in our case,  means there
is $\delta>0$ such that, for any $x\in K$ and any non-trivial
$E^c$-segment $I$ based on $x$, there is $m\ge 1$ such that
$l(f^m(I))>\delta$.

For any $x\in K$, denote
$$W_\gamma^{+c}(x)=\phi^c(x)[0,\gamma], ~~~~ W_\gamma^{-c}(x)=\phi^c(x)[-\gamma,0].$$

\vskip 0.3cm
 {\noindent\bf Claim 3.}
There is $\gamma\in(0, \delta)$ such that,  for any $x\in K$,
$l(f^n(W_\gamma^{\pm c}(x)))\to 0$ as $n\to+\infty$, where
$W_\gamma^{\pm c}(x)$ means ``$W_\gamma^{+c}(x)$ and
$W_\gamma^{-c}(x)$".

This means $W_\gamma^{c}(x)\subset W^s(x)$ for any $x\in K$. (Recall
$W_\gamma^{c}(x) =W_\gamma^{+c}(x)\cup W_\gamma^{-c}(x)$.) Thus
Claim 3 says that, in dimension 1, sensitive dependence on initial
conditions for $f^{-1}$ with one side neighborhoods implies uniform
size of stable manifolds for $f$ with two sides neighborhoods. The
converse is obvious.

We first prove there is $\gamma\in(0, \delta)$ such that
$$l(f^n(W^{\pm c}_\gamma(x)))\le \delta$$ for any $x\in K$
and $n\geq 0$. Suppose for the contrary, for any $\gamma>0$ there
exists a point $x\in K$ and a positive integer $n_x$ such that
$$l(f^{n_x}(W_\gamma^{+(or -)c}(x)))> \delta.$$ Without loss of
generality, we assume that
$$l(f^{n_x}(W_\gamma^{+(or -)c}(x)))=
\delta$$ but $$l(f^k(W_\gamma^c(x)))< \delta$$ for any $0\leq
k<n_x$. It is easy to see that $n_x\to+\infty$ as $\gamma\to 0$. We
may assume $f^{n_x}(W_\gamma^{+(or -)c}(x))$ accumulate to an arc
$I$ and $f^{n_x}(x)$ accumulate to some point $y\in K$. It is easy
to check that $I$ is a $\delta$-$E^c$ segment, contradicting the
assumption.

Now we prove for this $\gamma\in(0, \delta)$ and any $x\in K$,
$$l(f^n(W_\gamma^{\pm c}(x)))\to 0$$ as $n\to+\infty$. Suppose
there is  $x\in K$ such that
$$l(f^n(W_\gamma^{+(or -) c}(x)))\nrightarrow 0.$$ Then there exist
$\eta>0$ and a sequence of positive integers $n_k$ such that
$$l(f^{n_k}(W_\gamma^{+(or -) c}(x)))>\eta.$$ We may assume
$f^{n_k}(W_\gamma^{+(or -) c}(x))\to J$ and $f^{n_k}x\to z\in K$.
Then $J$ is non-trivial. Since
$$l(f^n(W_\gamma^{\pm c}(x)))\le \delta$$
for any $n\ge 0$, it is easy to see that $J$ is a $\delta$-$E^c$
segment, contradicting the assumption. This ends the proof of Claim
3.
\end{proof}

Thus $W_\gamma^{c}(x)\subset W^s(x)$ for any $x\in K$. Since $K$ is
a minimal set, by Item (5) of Proposition \ref{Generic}, there are
periodic orbits $P_n$ that approach $K$ in the Hausdorff metric.
Take $p_n\in P_n$ such that $p_n\to y\in K$. We may assume that
$W^u_{loc}(p_n)$ and $W^u_{loc}(y)$ both have size at least $\gamma$
and we simply denote them to be $W^u_{\gamma}(p_n)$ and
$W^u_{\gamma}(y)$. Note that, while $W_\gamma^{c}(y)\subset W^s(y)$,
it is not known if $W_\gamma^{c}(p_n)\subset W^s(p_n)$. What is
known is that, being a periodic interval of a Kupka-Smale system
$f$, $W_\gamma^{c}(p_n)$ has at most finitely many periodic points
of the same (or doubled, depending on $f^{\pi(p_n)}$ preserves or
flips the orientation of the interval) period as that of $p_n$.
Since $E^u$ is expanding, these periodic points have index 0 and 1,
alternately. Now we assume $p_n$ and $y$ are close enough so that
$$W^{u}_{\gamma}(p_n)\cap W_{\gamma}^c(y)\neq\emptyset, ~~~~
W^{u}_{\gamma}(y)\cap W_{\gamma}^c(p_n)\neq\emptyset.$$ We may
assume $\gamma$ was chosen small enough so that each intersection
contains a single point. Let $z$ be the unique point in
$W^{u}_{\gamma}(y)\cap W_{\gamma}^c(p_n).$ On the $E^c$-interval
$[p_n, z]$, the periodic point $q$ that is closest to $z$ can not be
a source, because the unstable manifolds of $y$ and $q$ can not
intersect. That is, $q$ must be a saddle of index 1 and $[q,
z]\subset W^s(q)$. Note that $W^{u}_{\gamma}(q)\cap
W_{\gamma}^c(y)\neq\emptyset$. In summary,
$$W^{u}_{\gamma}(q)\cap W_{\gamma}^s(y)\neq\emptyset, ~~~~
W^{u}_{\gamma}(y)\cap W^s(q)\neq\emptyset.$$ Thus
 $y$ and
$q$ are in the same chain class. This proves Theorem \ref{mainprop}
in Case 2, hence ends the proof of Theorem \ref{mainprop}.

\section {Proof of Theorem A}

In this section we complete the proof of Theorem A. First we prove
it for a generic $f$.

\begin{Proposition}\label{Mainprop2}
Let $f\in\mathcal{R}$, and let $C_f(p)$ be a structurally stable
chain class of $f$. If  ${\rm Ind}(p)=1$ or $\dim M-1$, then
$C_f(p)$ is hyperbolic.
\end{Proposition}

\begin{proof}
 Let $$T_{C_f(p)}M=E\oplus F$$ be the dominated splitting
given right before Proposition \ref{Proposition31}. We take the case
${\rm Ind}(p)=1$ and hence ${\rm dim}E=1$. The case ${\rm
Ind}(p)=\dim M-1$ can be treated similarly. We prove $E$ is
contracting.

Suppose $E$ is not contracting. Then there is a minimally
non-contracting set $\Lambda\subset C_f(p)$ of $E$. By Proposition
\ref{Proposition31}, $E\oplus F$ restricted to $\Lambda$ is
partially hyperbolic. More precisely, write
$$T_\Lambda M=E^c\oplus E^u, $$ where $E^c=E|_\Lambda$ and $E^u=F|_\Lambda$, then
$E^u$ is expanding, and
$$\lim_{n\to+\infty}\frac{1}{n}\sum_{i=0}^{n-1}\log\|Df^m|_{E^c(f^{im}x)}\|=0$$
for any $x\in \Lambda$.   Take any minimal set $K\subset \Lambda$.
$K$ must be non-trivial because otherwise, by the limit equality,
$K$ reduces to a non-hyperbolic periodic orbit, contradicting
$f\in\mathcal{R}$.  By Theorem \ref{mainprop}, there exist periodic
orbits $Q_n\subset C_f(p)$ such that $Q_n\to K$ in the Hausdorff
metric. By Proposition \ref{Proposition1}, each $Q_n$ is
homoclinically related to ${\rm Orb}(p)$. By Proposition
\ref{Proposition23}, there exist $\lambda\in(0,1)$ and a positive
integer $m$ such that
$$\prod_{i=0}^{k-1}\|Df^m|_{E^c(f^{im}(q))}\|<\lambda^k$$
for  $q\in Q_n$, where $k=[\pi(q)/m]$.  Note that since $K$ is
non-trivial and hence $\pi(Q_n)\to \infty$, by slightly enlarging
$\lambda$ if necessary we have put $C=1$ in the inequality. Take
$\lambda'\in (\lambda, 1)$. By Pliss's Lemma, there are $q_n\in Q_n$
such that
$$\prod_{i=0}^{j-1}\|Df^m|_{E^c(f^{im}(q_n))}\|<(\lambda')^j$$
for all $j\geq 1$ (note that, since $Q_n$ is periodic, $q_n$ can be
taken so that $j$ runs over all positive integers). Taking a
subsequence if necessary we assume $q_n\to x\in  K$. Then
$$\limsup_{n\to+\infty}\frac{1}{n}\sum_{i=0}^{n-1}\log\|Df^m|_{E^c(f^{im}(x))}\|<{\rm log}\lambda'.$$
This contradicts the above limit equality. Thus $E$ is  contracting.

 Note that, by Proposition
\ref{Proposition23}, $F$ is uniformly expanding at the periods for
all periodic points $q$ homoclinically related to $p$. (Here the
phrase ``uniformly expanding at the periods" means the inequality in
Theorem \ref{PLC}, as remarked after Theorem \ref{PLC}.) Thus
Proposition \ref{Mainprop2} follows directly from the following
proposition.
\end{proof}

\begin{Proposition} [\cite{BGY}] Let $f$ be a diffeomorphism and $p$
be a hyperbolic periodic point of $f$. Assume the homoclinic class
$H(p)=H(p, f)$ admits a dominated splitting $T_{H(p)}M = E\oplus F$
such that $E$ is contracting and $\dim(E)={\rm Ind}(p)$. If $F$ is
uniformly expanding at the periods for all periodic points $q$
homoclinically related to $p$, then $F$ is uniformly expanding on
$H(p)$.
\end{Proposition}

Now we prove Theorem A without assuming $f$ is generic. We argue
with the shadowing property.  We say that $C_f(p)$ is $C^1$-{\it
robustly shadowable} (in \cite{WGW} it is called {\it stably
shadowable}) if there exists a neighborhood $\mathcal{U}(f)$ of $f$
such that for any $g\in \mathcal{U}(f)$, $C_g(p_g)$ has the
shadowing property, where $p_g$ is the continuation of $p$.

\begin{Proposition}
Let $f\in {\rm Diff}(M)$. If $C_f(p)$ is structurally stable, then
$C_f(p)$ is $C^1$-robustly shadowable.
\end{Proposition}
\begin{proof}
Fix a $C^1$ neighborhood ${\mathcal U}$ of $f$ in ${\rm Diff}(M)$
such that, for any $g\in {\mathcal U}$, there is a homeomorphism $h$
that conjugates $C_f(p)$ and $C_g(p_g)$. Take $f_0\in\mathcal{R}\cap
{\mathcal U}$. Since $C_f(p)$ is structurally stable, so is
$C_{f_0}(p_{f_0})$. By Proposition \ref{Mainprop2},
$C_{f_0}(p_{f_0})$ is hyperbolic, hence  has the shadowing property.
For every $g\in {\mathcal U}$, since $C_g(p_g)$ is conjugate to
$C_{f_0}(p_{f_0})$, $C_g(p_g)$ has the shadowing property too. Thus
$C_f(p)$ is $C^1$-robustly shadowable, proving the proposition.
\end{proof}

Thus our main result, Theorem A, follows from the following result:

\begin{theorem}[\cite{WGW}]
If $C_f(p)$ is $C^1$-robustly shadowable, then $C_f(p)$ is
hyperbolic.
\end{theorem}

\end{document}